\documentclass[12pt]{amsart}
\usepackage{amssymb,latexsym}
\usepackage{enumerate}
\usepackage{hyperref}

\makeatletter \@namedef{subjclassname@2010}{%
  \textup{2010} Mathematics Subject Classification}
\makeatother

\newtheorem{Theorem}{Theorem}
\newtheorem*{MainTheorem}{Main Theorem}

\newtheorem{Lemma}{Lemma}
\newtheorem{Corollary}{Corollary}
\newtheorem{Problem}{Problem}
\newtheorem*{Definition}{Definition}

\newcommand{\FF}[0]{\mathbb F}

	\newcommand{\ZZ}[0]{\mathbb Z}

\newcommand{\cB}[0]{\mathcal B}	
	\newcommand{\cP}[0]{\mathcal P}
	\newcommand{\cQ}[0]{\mathcal Q}

\newcommand{\pp}[0]{\textbf{\textit{p}}}

\newcommand{\ee}[0]{\textbf{\textit{e}}}

\newcommand{\eps}[0]{\varepsilon}

\newcommand{\leg}[2]{\left(\frac{#1}{#2}\right)}

\renewcommand{\mod}[0]{\text{ mod }}

\newcommand{\lr}[1]{\left(#1\right)}

\renewcommand{\phi}{\varphi}

\begin{document}


\baselineskip=17pt



\title{The Least Number with Prescribed Legendre Symbols}

\author[Brandon Hanson]{Brandon Hanson} \address{Pennsylvania State University\\
University Park, PA}
\email{bwh5339@psu.edu}

\author[Robert C. Vaughan]{Robert C. Vaughan} \address{Pennsylvania State University\\
University Park, PA}
\email{rcv4@psu.edu}

\author[Ruixiang Zhang]{Ruixiang Zhang} \address{Princeton University\\
Princeton, NJ}
\email{ruixiang@math.princeton.edu}
\date{}

\begin{abstract}
In this article we estimate the number of integers up to $X$ which can be properly represented by a positive-definite, binary, integral quadratic form of small discriminant. This estimate follows from understanding the vector of signs that arises from computing the Legendre symbol of small integers $n$ at multiple primes.
\end{abstract}

\maketitle
\section{Introduction}

A well-known and outstanding problem in number theory is the estimation of the least quadratic non-residue modulo a prime $p$. Recall, the least quadratic non-residue modulo $p$ is the integer
\[n_p=\min\left\{n>1:\leg{n}{p}=-1\right\},\] where $\leg{x}{p}$ is the Legendre symbol of $x$ modulo $p$. Vinogradov conjectured that for any $\eps>0$, $n_p$ should be at most $p^\eps$ provided $p$ is sufficiently large relative to $\eps$. This conjecture is still open, however Linnik proved that any exceptions to it are sparse. Specifically, he proved the conjecture holds for all but $O(\log\log N)$ of the primes $p\leq N$, with the implied constant depending only on $\eps$.

Now suppose $p_1$ and $p_2$ are distinct, odd primes and $n$ is an integer not divisible by either. There are four possibilities for the vector $\lr{\leg{n}{p_1},\leg{n}{p_2}}$. How large must $N$ be so that all four vectors are realized by integers bounded by $N$? In general, one might ask the following.

\begin{Problem}\label{MainProblem}
Given distinct odd primes $p_1,\ldots,p_k$, how large must $N$ be before one has seen each of the $2^k$ different vectors of signs in $\{1,-1\}^k$ realized by a vector of the from
\[\lr{\leg{n}{p_1},\ldots,\leg{n}{p_k}}\] with $1\leq n\leq N$?
\end{Problem}
We believe that the above problem is intrinsically interesting, but there is further motivation studying it. Recall that the quadratic form $F(x,y)=Ax^2+Bxy+Cy^2$ has discriminant $d=B^2-4AC$, is said to represent $q$ if $F(x,y)=q$ has a solution $(x,y)\in\ZZ^2$ with $x$ and $y$, and is said to properly represent $q$ if furthermore $x$ and $y$ can be taken to be relatively prime. We say the form is definite if $d<0$. One could then ask, 

\begin{Problem}\label{LeastDiscriminant}
What is the least positive integer $d$ such that $q$ is represented by a quadratic form of discriminant $-d$?
\end{Problem}
If we allow for indefinite forms, which is to say forms with positive discriminant, then this problem is less interesting. A number is a difference of squares if and only if it is not congruent to $2 \mod 4$. Thus, either $q$ is not congruent to $2\mod 4$ and $q=x^2-y^2$ has a solution, or else $q/2$ is not congruent to $2\mod 4$ in which case $q=2x^2-2y^2$ has a solution. 

As is outlined below, answering Problem \ref{LeastDiscriminant} essentially amounts to the solution of Problem \ref{MainProblem}, because of the following theorem (see for instance \cite[Proposition 4.1]{B}).
\begin{Theorem}\label{Discriminant}
The number $q$ is properly represented by a binary quadratic form of discriminant $d$ if and only if $d$ is a square modulo $4q$.
\end{Theorem}

When $q$ is an odd prime and $q\equiv 3\mod 4$, Problem \ref{LeastDiscriminant} boils down to that of finding the least quadratic non-residue modulo $q$. But when handling Problem \ref{LeastDiscriminant} for general $q$, rather than have $\leg{-d}{q}=-1$, we require that $\leg{d}{p}=(-1)^\frac{p-1}{2}$ for each odd prime $p$ dividing $q$. Thus we are interested in prescribing the Legendre symbol of $d$ at several primes, which returns us to Problem \ref{MainProblem}. 

One goal of this article is to extend Linnik's result on the least non-residue and show that one can usually prescribe the sign of the Legendre symbol simultaneously at many primes with a small integer, barring some ``local'' obstructions as described in the next section.    

\section{Notation, preliminary observations, and statement of results}

To begin, let $k\geq 1$ be an integer and $p_1,\ldots,p_k\leq N$ be distinct odd primes. Let $\pp=(p_1,\ldots,p_k)$, $q=p_1\cdots p_k$ and write \[\leg{n}{\pp}=\lr{\leg{n}{p_1},\ldots,\leg{n}{p_k}}.\] Finally, denote by $G=\{\pm 1\}^k$ the (multiplicative) group of all $k$-tuples with entries $\pm 1$. As described in the introduction, we are interested in the number
\[n_q=\max_{\ee\in G}\min\left\{n:n\geq 1,\ \leg{n}{\pp}=\ee,\ n\equiv 1\mod 8\right\}.\] Thus $n_q$ is the least positive integer such that we observe all possible sign choices for $\leg{n}{\pp}$ with integers less than $n_q$ and congruent to $1\mod 8$. The congruence condition allows us to further insist that we deal only with squares modulo a power of two, which is needed in applications. 

It is convenient to identify the group $G$ with $\FF_2^k=(\ZZ/2\ZZ)^k$, the vector space of dimension $k$ over the field $\FF_2$, in the natural way. To be concrete, set \[U_q(y)=\{n:1\leq n\leq y,\ (n,q)=1,\ n\equiv 1\mod 8\}.\] Consider the map $\theta_q:U_q(y)\to \FF_2^k$ given by \[\theta_q(n)=\lr{\frac{1}{2}\lr{1-\leg{n}{p_1}},\ldots,\frac{1}{2}\lr{1-\leg{n}{p_k}}}\mod 2.\] The $i$'th entry of $\theta_q(n)$ is $1\mod 2$ if $n$ is a quadratic non-residue modulo $p_i$ and $0\mod 2$ if $n$ is a quadratic residue modulo $p_i$. The map $\theta_q$ is also an additive function in the sense that $\theta_q(mn)=\theta_q(m)+\theta_q(n)$, with the addition operation belonging to $\FF_2^k$. Moreover, for $y$ sufficiently large the map $\theta_q$ is surjective by Chinese Remainder Theorem and the fact that the primes $p_i$ are distinct and odd.

Suppose we know that the integers in $U_q(y)$ span $\FF_2^k$, in the sense that $\left\{\theta_q(n):n\in U_q(y)\right\}$ contains a basis of $\FF_2^k$. Then $n_q\leq y^k$. Indeed, any vector $\ee\in \FF_2^k$ is the sum of at most $k$ basis vectors, each of which is of the form $\theta_q(m)$ with $m\leq y$. The product of these integers $m$ gives an integer $n$ with $\theta_q(n)=\ee$. It is therefore sensible to consider the number
\[g_q=\min\left\{y:\left\{\theta_q(n):n\in U_q(y)\right\}\text{ spans }\FF_2^k\right\}.\] We record the above observation as a lemma.

\begin{Lemma}\label{Generation}
Let $q=p_1\cdots p_k$ be an odd, square-free integer. Then $n_q\leq g_q^k$.
\end{Lemma}

We have so far reduced the problem of bounding $n_q$ to that of bounding $g_q$. In the spirit of Linnik, we would like to estimate the number of $q$ for which $g_q$ is large. However, we need to be mindful of the following obstruction. If $g_d>y$ for some divisor $d$ of $q$ then $g_q>y$ as well. Thus, if $d$ is small and $g_d>y$, then $g_q>y$ for at least $[Q/d]$ numbers up to $Q$ (the multiples of $d$), which is substantial. So, we need to restrict our attention to what we will call \emph{eligible} $q$, namely those which are not divisible by some $d$ for which $g_d$ is large. In fact, for technical reasons, we have need to consider the number $g_{q,r}$ defined as
\[g_{q,r}=\min\left\{y:\left\{\theta_q(n):n\in U_{qr}(y)\right\}\text{ spans }\FF_2^k\right\}.\]
Notice that in this latter definition, we want to generate the full group of signs with numbers not just coprime to $q$ but also to $r$. 

\begin{Definition}
Let $y\geq 1$ be a parameter. We say $q$ is $y$-eligible if for each divisor $d$ of $q$ with $1<d<q$, we have $g_{d,q}\leq y$. Otherwise, we say $q$ is $y$-ineligible. If $g_q>y$ and $q$ is $y$-eligible, we say $q$ is $y$-exceptional.
\end{Definition}

So, an odd prime $p$ is always $y$-eligible, and is $y$-exceptional if $n_p>y$. It also becomes clear why we need to introduce the the notion of $g_{q,r}$: it may be that $g_d\leq y$ for each proper divisor $d$ of $q$, but in order to get this full set of generators, we must use numbers which are coprime to $d$ but not coprime to $q$. 

We now state our main results. 

\begin{MainTheorem}\label{LogCase}
Let $a\geq 3$ be fixed. Suppose $\cQ(Q,a)$ is the set of all integers $q\leq Q$ which are odd, square-free, and $(\log q)^a$-exceptional. Then, for $\delta>0$, we have \[\cQ(Q,a)\ll_{\delta,a} Q^{2/a+\delta}.\]
\end{MainTheorem}

Recall that the square-free radical of $q$ is $r=\prod_{p|q}p$.

\begin{Corollary}\label{Counting}
Let $\eps\in(0,1)$. There are positive numbers $a$ and $c$ which depend only on $\eps$ and such that the following holds. For $Q$ sufficiently large in terms of $\eps$, there are at least $c Q^\eps(\log\log Q)^{-1}$ numbers $q\in [Q,Q+Q^\eps]$ such that $g_{r}\leq (\log r)^a$, where $r$ is the square-free radical of $q$.
\end{Corollary}

As applications of the Main Theorem, we have the following corollaries. 

\begin{Corollary}\label{QuadraticForm}
Let $\eps>0$ and let $Q$ be sufficiently large in terms of $\eps$. There is an integer in the interval $[Q,Q+Q^\eps]$ which is properly represented by a definite, binary quadratic form of discriminant $-d$ with $d\leq Q^\eps$. 
\end{Corollary}

\begin{Corollary}\label{SumsOfSquares}
Let $\eps>0$ and let $Q$ be sufficiently large in terms of $\eps$. There is an integer $q$ in the interval $[Q,Q+Q^\eps]$ which can be written as \[q=\frac{1}{u}x^2+\frac{v}{u}y^2\] with $0\leq u,v\leq Q^\eps$ and $u\neq 0$.
\end{Corollary}

This final corollary can be compared with the problem of bounding the gaps between consecutive sums of two squares. Being that there are about $x(\log x)^{-1/2}$ integers up to $x$ which are a sum of two squares, one might expect that the gaps between such integers are at most $(\log x)^c$ for some positive constant $c$. However the best known bound, due to Bambah and Chowla \cite{BC}, is that there is a integers between $x$ and $x+O(x^{1/4})$ which is a sum of two squares. Corollary \ref{SumsOfSquares} says that one has much smaller gaps if we weaken squares to numbers which are in a sense ``almost-squares''.

\section{Facts from analytic number theory}
Here we recall some required background results from analytic number theory. 
Let \[S(x,y)=\{n\leq x: p|n\implies p\leq y\text{ and }p=1\mod 8\}\] and for a positive integer $q$, let \[S_q(x,y)=\{n\in S(x,y): (n,q)=1\}.\]
We need to estimate these sets. To begin, we have the following which is a modification of Corollary 7.9 from \cite{MV}.
\begin{Theorem}\label{MV}
Suppose $a$ is in the range $2\leq a<(\log x)^{1/2}/(2\log\log x)$ and let $\delta>0$. Then for $x$ sufficiently large and $q\leq x$,
\[x^{1-1/a-\delta}\ll_{\delta,a} |S_q(x,(\log x)^a)|\leq |S(x,(\log x)^a)|\ll_{\delta,a} x^{1-1/a+\delta}.\]
\end{Theorem}
\begin{proof}
Let $y\geq (\log x)^2$ and let $\cP=\{p_1,\ldots,p_T\}$ denote the set of primes $p\leq y$ with $p=1\mod 8$ and which do not divide $q$. Any product of primes in $\cP$ which does not exceed $x$ belongs to $S_q(x,y)$. By the Prime Number Theorem in arithmetic progressions (Corollary 11.21 in \cite{MV}), we have \[\pi(y;8,1)=\frac{y}{\log y}\lr{\frac{1}{4}+o(1)}.\] Since $\omega(n)$ is maximized when $n$ is a primorial number \[\omega(q)\ll \frac{\log q}{\log\log q}\leq \frac{\log x}{\log\log x}.\]
Because $y\geq (\log x)^2$, for $x$ sufficiently large, $T\geq\frac{y}{5\log y}$.
Now consider any product of primes in $\cP$. Its logarithm is of the form
\[\sum_{j=1}^T v_j\log p_j\leq \log y\sum_{j=1}^T v_j.\] Thus a lower bound for $|S_q(x,y)|$ is the number of vectors \[N=\left|\left\{(v_1,\ldots,v_T)\in \ZZ^T:v_j\geq 0,\ \sum_{j=1}^T v_j\leq \left[\frac{\log x}{\log y}\right]\right\}\right|.\] Letting $U=\left[\frac{\log x}{\log y}\right]$, it is a simple combinatorial argument (see Lemma 7.7 of \cite{MV}) that the number of such vectors is \[N=\binom{T+U}{T}.\] By Stirling's Formula, 
\[n!\asymp n^{n+1/2}e^{-n}\] so that 
\begin{align*}N&\gg\frac{(U+T)^{U+T+1/2}e^{-U-T}}{U^{U+1/2}e^{-U}T^{T+1/2}e^{-T}}\\
&\geq\lr{\frac{U+T}{U}}^U\frac{1}{\sqrt{U}}.\end{align*}
The right hand side above is increasing in $T$, thus
\[N\gg \lr{1+\frac{y}{5U\log y}}^U\frac{1}{\sqrt{U}}.\] 
If $u=\frac{\log x}{\log y}$ then $u-1\leq U\leq u\leq \frac{y}{\log y}$. 
Thus 
\[N\gg\lr{\frac{y}{5u\log y}}^{u-1}\frac{1}{\sqrt{u}}.\]
For the exponent $-1$ and the factor $1/\sqrt u$ we note that
\[\frac{1}{\sqrt u}\frac{5u\log y}{y}\geq \frac{1}{y}\] and thus
\[N\gg \frac{1}{y}\lr{\frac{y}{5\log x}}^{\frac{\log x}{\log y}}=\frac{x}{y}\exp\lr{-\frac{\log x}{\log y}\log(5\log x)}\]
Now we take $y=(\log x)^a$ where $2\leq a<(\log x)^{1/2}/(2\log \log x)$.
Then we get
\[N\gg x^{1-\frac{1}{a}}\exp\lr{-a\log\log x-\frac{\log 5\log x}{a\log\log x}}\gg_{\delta,a} x^{1-1/a-\delta}.\]
For the upper bound, trivially $|S_q(x,(\log x)^a)|$ is less than $|S(x,(\log x)^a)|$, which is in turn at most the number of $(\log x)^a$-smooth numbers up to $x$. There are at most $O_{\delta,a}(x^{1-1/a+\delta})$ such numbers by Corollary 7.1 of \cite{MV}.
\end{proof}

The main ingredient we need for the proof of our main theorems is a Large Sieve inequality. This one can be found in \cite[Theorem 7.13]{IK}.
\begin{Theorem}[Large Sieve]\label{LargeSieve}
Let $Q\geq 1$ and let $(a_n)_{n\leq x}$ be a sequence of complex numbers. Then
\[\sum_{q\leq Q}\frac{q}{\phi(q)}\sideset{}{'}\sum_{\chi\mod q}\left|\sum_{n\leq x} a_n\chi(n)\right|^2\ll (Q^2+x)\sum_{n\leq x} |a_n|^2.\]
In the above sum over $\chi$ we mean that the summation occurs over all primitive characters $\chi$ of modulus $q$.
\end{Theorem}

In order to make Lemma \ref{Generation} useful, we need to have integers with a reasonable number of prime factors. The next lemma tells us such integers are ubiquitous.

\begin{Lemma}\label{PrimeFactors}
Let $\eps\in(0,1)$ and $K\geq 1$ be fixed. For all $Q$ sufficiently large in terms of $\eps$, there are at most $O\lr{Q^\eps(\log\log Q)^{-K}}$ integers $q\in[Q, Q+Q^\eps]$ with at least $(\log\log Q)^{K+1}$ prime factors.
\end{Lemma}
\begin{proof}
Let $u\in [1,Q]$ be a number to be determined later. Write \[\omega_u(q)=\sum_{\substack{p|q\\ p\leq u}}1.\]
Then \[\omega(q)-\omega_u(q)=\sum_{\substack{p|q\\ p> u}}1\leq \frac{1}{\log u}\sum_{p|q}\log p\leq\frac{\log q}{\log u}.\]
The number of integers $q$ in question is at most
\begin{align*}
&(\log\log Q)^{-K-1}\sum_{Q\leq q\leq Q+Q^\eps}\omega(q)\\
\leq&(\log\log Q)^{-K-1}\sum_{Q\leq q\leq Q+Q^\eps}\omega_u(q)+2Q^\eps(\log\log Q)^{-K-1}\frac{\log Q}{\log u}.
\end{align*}
To estimate the sum,
\[\sum_{Q\leq q\leq Q+Q^\eps}\omega_u(q)\leq \sum_{p\leq u}\frac{Q^\eps}{p}+O(u)\ll Q^\eps\log\log Q+O(u).\] Taking $u=Q^\eps$ gives the bound
\[O\lr{\frac{Q^\eps}{(\log\log Q)^K}+\frac{Q^\eps}{\eps(\log\log Q)^{K+1}}}=O\lr{\frac{Q^\eps}{(\log\log Q)^K}}\] once $Q$ is sufficiently large in terms of $\eps$.
\end{proof}

Finally, in order to guarantee that we can find $y$-eligible numbers in short intervals, we will need a basic consequence of Brun's Pure Sieve. This result can be read from Corollary 6.2 in \cite{FI}.
\begin{Theorem}\label{BrunSieve}
Let $a\geq 1$ and $\eps>0$ be fixed. Then the number of integers in the interval $[Q,Q+Q^\eps]$ with no prime divisor less than $z$ is asymptotic to $Q^\eps V\lr{z}$, where $V(z)=\prod_{p<z}(1-p^{-1})$, provided $z$ is in the range $1\leq z\leq (\log Q)^a$.
\end{Theorem}

\section{Proofs of Main Theorems and Corollaries}

The following lemma provides the key reduction to a problem which is approachable by the Large Sieve.

\begin{Lemma}\label{Subspace}
Let $q$ be an odd, square-free, and $y$-exceptional number. Then for any $n\in S_q(x,y)$, we have \[\leg{n}{q}=\prod_{p|q}\leg{n}{p}=1.\]
\end{Lemma}
\begin{proof}
For convenience, write $k=\omega(q)$, write $q=p_1\cdots p_k$, and denote by \[\Theta_q(y)=\theta_q(U_q(y))\] the image of $U_q(y)$ in $\FF_2^k$. Since $q$ is $y$-exceptional we have $g_q>y$, and so $\Theta_q(y)$ spans a proper subspace of $\FF_2^k$, say $H$. A number $n\in S_q(x,y)$ is a product of primes in $U_q(y)$. Thus, writing $n=\prod_{p|n}p^{v_p}$, we have \[\theta_q(n)=\sum_{p} v_p\theta_q(p)\in H\] since $\theta_q(p)\in H$ for each $p$ occurring in the sum. Since $H$ is a proper subspace, there is a non-zero vector in $H^\perp$. In other words, for some non-empty subset $I\subseteq\{1,\ldots,k\}$, each vector $(x_1,\ldots,x_k)\in H$ satisfies \[\sum_{i\in I}x_i=0.\] In fact, we must have $I=\{1,\ldots,k\}$. To see this, suppose that $I$ were a proper subset with $|I|=l<k$. Let $d_I=\prod_{i\in I}p_i$. Then the projection $\pi_I:H\to \FF_2^l$ given by \[\pi_I(x_1,\ldots, x_k)=(x_i)_{i\in I}\] is not surjective. Indeed, any vector in the image has co-ordinates which sum to $0\mod 2$. But this projection contains the image $\theta_{d_I}(U_q(n))$. So the proper divisor $d_I$ satisfies $g_{d_I,q}>y$, which violates the $y$-eligibility of $q$. Thus the each element of $H$ satisfies $\sum_{i=1}^k x_i=0$ which means that \[\prod_{p|q}\leg{n}{p}=1\] for any $n\in S_q(x,y)$.
\end{proof}

We now prove our main theorem. The proof boils down to Linnik's result on the least number $n_\chi$ for which a primitive character $\chi$ satisfies $\chi(n_\chi)\neq 1$. The result is stated but not proved, in \cite{DK}. A proof is given in \cite{Pol} (Lemma 5.3), and we will will follow in much the same manner. 

\begin{proof}[Proof of Main Theorem]
Fix $\delta>0$. Let $C_{a,\delta}$ be a constant depending only on $a$ and $\delta$ which is at our disposal, and let $x>C_a$. We will work dyadically, and apply Theorem \ref{LargeSieve} with \[a_n=\begin{cases}1&\text{ if }n\in S(x^2,(\log x)^a)\\0&\text{ otherwise.}\end{cases}\] 

Suppose $q\in\cQ(Q,a)$ is in the range $x<q\leq 2x$. Then \[S(x^2,(\log x)^a)\subseteq S(x^2,(\log q)^a).\] So if $n\in S(x^2,(\log x)^a)$, then by Lemma \ref{Subspace},
\[\prod_{p|q}\leg{n}{p}=\begin{cases}1&\text{ if }(n,q)=1\\ 0&\text{ otherwise}.\end{cases}\] It follows that
\[\sum_{n\leq x^2}a_n\leg{n}{q}\geq |S_q(x^2,(\log x)^a)|\]

Suppose $x$ is large enough so that it satisfies $2^{-a}>\log(x^2)^{-\delta/16}$, which will be guaranteed by increasing $C_{a,\delta}$ as necessary. By Theorem \ref{MV}, 
\begin{align*}
|S_q(x^2,(\log x)^a)|&=|S_q(x^2,2^{-a}(\log x^2)^a)|\\
&\geq|S_q(x^2,(\log x^2)^{a-\delta/16})| \\
&\gg_{\delta,a}x^{2-2/(a-\delta/16)-\delta/16}\\
&\gg_{\delta,a}x^{2-2/a-\delta/8}
\end{align*} 
by increasing $C_{a,\delta}$ if necessary.
Now we sum over all $q\in \cQ_x=\cQ(Q,a)\cap (x,2x]$ to get
\begin{align*}\sum_{q\in\cQ_x} x^{4-4/a-\delta/4}&\ll_{\delta,a}\sum_{q\in\cQ_x}\frac{q}{\phi(q)}\sideset{}{'}\sum_{\chi\mod q}\left|\sum_{n\leq x^2} a_n\chi(n)\right|^2\\
&\ll_{\delta,a} x^2|S\lr{x^2,(\log x^2)^a}|\\
&\ll_{\delta,a} x^{4-2/a+\delta/4}.
\end{align*} 
Rearranging, we see that
\[|\cQ_x|\ll_{\delta,a} x^{2/a+\delta/2}\leq Q^{2/a+\delta/2}\] for $x\leq Q$. Summing
over all $x=2^j$ in the range $\log_2(C_{a,\delta})\leq j\leq \log_2 Q$ we get
\[|\cQ(Q,a)|\ll_{\delta,a}\sum_{j \leq  \log Q}|\cQ_{2^j}|\ll_{\delta,a} Q^{2/a+\delta}.\] Here we have used that there is a contribution of $O_{a,\delta}(1)$ from the terms with $2\leq q\leq \log_2(C_{a,\delta})$ and the fact that $\log Q\ll_\delta Q^{\delta/2}$.
\end{proof}

Before proving the first corollary, we need some lemmas concerning eligibility.

\begin{Lemma}\label{Exceptional}
If $q$ is odd, square-free and $y$-ineligible, and if the least prime dividing $q$ is $p$, then $q$ has a proper divisor $d$ which is $\min\{y,p-1\}$-exceptional. In particular, if $q$ is odd, square-free and $y$-ineligible, and if all prime factors of $q$ exceed $y+1$, then $q$ has a proper divisor $d$ which is $y$-exceptional.
\end{Lemma}
\begin{proof}
Since $q$ is $y$-ineligible it has a divisor $d$ such that $g_{d,q}>y$. Let \[d_0=\min\{d:d|q,\ 1<d<q,\ g_{d,q}>y\}.\] Then $d_0$ is $y$-eligible. Indeed, if $d_0$ were not $y$-eligible, it would have a proper divisor $d'$ for which $g_{d',d_0}>y$. But then, since $d_0|q$, we have $g_{d',q}\geq g_{d',d_0}>y$ contradicting the minimality of $d_0$. So $d_0$ is $y$-eligible but $g_{d_0,q}>y$. Either $g_{d_0}>y$ as well and so $d_0$ is $y$-exceptional, or else $g_{d_0}\leq y$. In the latter case, the vectors $\theta_{d_0}(n)$ with $n\in U_{d_0}(y)$ generate the full group of signs for $d_0$, but those with $n\in U_q(y)$ do not. However, all of the elements of $U_{d_0}(y)\setminus U_q(y)$ are divisible by some prime which is at least as big as $p$, and so $g_{d_0}>p-1$. It may now be the case, since $p-1<y$, that $d_0$ is not $p-1$-eligible. If so, let \[d_1=\min\{d:d|d_0,\ 1<d<d_0,\ g_{d,d_0}>p-1\}.\] As before, $d_1$ is $p-1$-eligible and $g_{d,d_0}>p-1$. But in fact, since $d_0$ is only divisible by primes greater than $p$, $U_{d_0}(p-1)=U_{d_1}(p-1)$, and $g_{d_1}>p-1$. Hence $d_1$ is $p-1$-exceptional. The second statement of the lemma follows immediately.
\end{proof}

\begin{Lemma}\label{Exceptional2}
If $q$ is an odd, square-free integer which is $(\log q)^a$-ineligible, and if all prime factors of $q$ are at least $2(\log q)^a$, then $q$ has a divisor $d$ which is $(\log d)^a$-exceptional.
\end{Lemma}
\begin{proof}
Since $q$ is not $(\log q)^a$-eligible then it has a divisor $d_1$ which is $(\log q)^a$-exceptional, by Lemma \ref{Exceptional}. Either $d_1$ is also $(\log d_1)^a$-exceptional (and we are done) or else it is not $(\log d_1)^a$-eligible. In the latter case it has a proper divisor $d_2$ which is either $(\log d_2)^a$-exceptional, or else not $(\log d_2)^a$-eligible. Continuing in this fashion, we must arrive at a divisor $d$ of $q$ which is $(\log d)^a$-exceptional. Indeed, eventually we would either stop or arrive at a prime $p$, and this prime is $(\log p)^a$-exceptional since all primes are $y$-eligible for all $y>1$.
\end{proof}

\begin{proof}[Proof of Corollary \ref{Counting}]
Assume $\eps<1/2$ without any loss of generality and let $a$ be a positive number at our disposal, which will depend only on $\eps$. Let $\cB$ be the set of all numbers in $[Q,Q+Q^\eps]$ which have no prime factors smaller than $2(\log Q)^a$. By Theorem \ref{BrunSieve}, we have \[|\cB|\sim Q^\eps\cdot V(2(\log Q)^a).\] Now, by Mertens' theorem,
\[-\log V(z)=\sum_{p<z}-\log(1-p^{-1})\leq \sum_{p<z}\frac{1}{p-1}=\log\log z+O(1).\] So
\[V(2(\log Q)^a)\gg \frac{1}{a\log\log Q}\] and hence
\[|\cB|\gg \frac{Q^\eps}{a\log\log Q}.\]

Let $\cB'\subseteq \cB$ be the subset of all elements of $q\in \cB$ which have a square-free radical $r=\prod_{p|q}p$ satisfying $g_r>(\log r)^a$. For such $q$, $r$ is either $(\log r)^a$-exceptional or $(\log r)^a$-ineligible. Now, all prime factors of $r$ exceed $2(\log r)^a$. So, if $r$ is $(\log r)^a$-ineligible, it follows from Lemma \ref{Exceptional2} that $r$ is divisible by a number $d$ which is $(\log d)^a$-exceptional. In either case, $q$ has a divisor $d$ which is square-free and $(\log d)^a$-exceptional. Moreover, this divisor $d$ satisfies $d\geq (\log Q)^a$ since it is a non-empty product of primes exceeding $(\log Q)^a$. Hence every integer in $\cB'$ is divisible by some $d\in \cQ(2Q,a)$ which is at least $(\log Q)^a$. It follows that the size of $\cB'$ is at most
\[\sum_{\substack{d\geq (\log Q)^a\\ d\in \cQ(2Q,a)}}\lr{\frac{Q^\eps}{d}+O(1)}\ll Q^\eps\int_{(\log Q)^a}^{2Q}\frac{\cQ(u,a)}{u^2}du+O\lr{\cQ(2Q,a)}\] by partial summation. We apply our Main Theorem with $\delta=\eps/4$. The integral is at most
\[O_{\eps}(1)\cdot Q^\eps \int_{(\log Q)^a}^\infty u^{2/a-2+\eps/4}du\ll_{\eps,a}\frac{Q^\eps}{(\log Q)^{a(1-\eps/4)-2}}\ll_{\eps, a}\frac{Q^\eps}{(\log Q)^{a/2}},\]
while
\[\cQ(2Q,a)\ll_\eps (2Q)^{2/a+\eps/4}\ll_{\eps} Q^{\eps/2}\] for $a$ sufficiently large in terms of $\eps$.
Thus once $a$ is sufficiently large in terms of $\eps$ and $Q$ is sufficiently large in terms of $a$, we have $|\cB\setminus \cB'|\gg_{\eps}\frac{Q^\eps}{\log\log Q}$. The integers in $\cB\setminus \cB'$ all have a square-free radical $r$ with $g_r\leq(\log r)^a$.
\end{proof}

\begin{proof}[Proof of Corollary \ref{QuadraticForm}]
We can assume $\eps$ is small without any loss of generality. By Theorem \ref{Discriminant}, it is enough to find some number $q\in[Q,Q+Q^\eps]$ and a discriminant $-d$ which is a square modulo $4q$. If $(d,4q)=1$, then in order for $-d$ to be a square modulo $4q$, it suffices that
\begin{enumerate}
\item $d\equiv 7\mod 8$,
\item $\leg{d}{p}=1$ for $p|q$ and $p\equiv1\mod 4$,
\item $\leg{d}{p}=-1$ for $p|q$ and $p\equiv3\mod 4$.
\end{enumerate}

By Corollary \ref{Counting}, there is a positive number $a$ such that the number of integers in $q\in[Q,Q+Q^\eps]$ with $g_q\leq (\log q)^a$ is at least
\[c_\eps Q^\eps(\log\log Q)^{-1}\] for some constant $c_{\eps}$ depending only on $\eps$.  By Lemma \ref{PrimeFactors}, one of these numbers will have $\omega(q)\leq (\log\log Q)^K$ for some $K$ sufficiently large in terms of $\eps$. Let $p_0$ be the smallest positive prime which is congruent to $7$ modulo $8$ and which does not divide $q$. Then, since $q$ has $k=\omega(q)\leq (\log\log Q)^K$ prime factors, $p_0$ is at most the $k+1$'th prime congruent to $7\mod 8$ which is at most $O(k\log k)=O_\eps(Q^{\eps/2})$ by the Prime Number Theorem. Since 
\[n_q\leq g_q^{\omega(q)}\leq (\log q)^{a(\log\log q)^K}=O_\eps(Q^{\eps/2}),\] we can find an integer $d_0\ll_{\eps} Q^{\eps/2}$ which is relatively prime to $q$, congruent to $1$ modulo $8$ so that $d_0p_0=7\mod 8$, and such that $\leg{d_0p_0}{p}$ is prescribed as needed for each $p$ dividing $q$. The number $d=d_0p_0$ satisfies the desired properties and the corollary is proved.
\end{proof}

Recall that a binary quadratic form $q(x,y)=Ax^2+Bxy+Cy^2$ of discriminant $-d$ is called \emph{reduced} if $|B|\leq A\leq C$. In this case $d=4AC-B^2\geq 3AC$, so that all coefficients are bounded by $d$. We say two quadratic forms $q_1(x,y)$ and $q_2(x,y)$ are equivalent if one can be obtained from the other by an invertible, integral change of variables. To be precise, we have $q_2(\alpha x+\beta y,\gamma x+\delta y)=q_1(x,y)$ for some integers $\alpha,\beta,\gamma,\delta$ with $\alpha \delta-\beta \gamma=1$. It is clear that equivalent forms represent the same numbers. This, combined with the following theorem, \cite[Theorem 2.3]{B}, shows that there is no loss of generality in working with reduced forms. 

\begin{Theorem}
Every binary quadratic form of discriminant $-d$ is equivalent to a reduced form of the same discriminant.
\end{Theorem}

We are now ready to prove our final corollary.

\begin{proof}[Proof of Corollary \ref{SumsOfSquares}]
By Corollary \ref{QuadraticForm}, we can represent an integer $q$ in the range $Q\leq q\leq Q+Q^\eps$ by some positive definite binary quadratic form $Q$ of discriminant $-d$ with $d$ at most $Q^\eps$. Without loss of generality we may assume this form is reduced, and thus write $q=Ax^2+Bxy+Cy^2$ with $B^2-4AC=-d$ for some $A,B,C$ bounded in absolute value by $Q^\eps$. It follows that $A,C>0$, and by completing the square we see \[q=\frac{1}{4A}\lr{(2Ax+By)^2+dy^2}.\] 
\end{proof}

\section*{Acknowledgment}
We thank the anonymous referee for pointing out errors in the original draft of this article, and for helpful comments on the exposition. Part of this work was carried out while the first and third authors were attending the IPAM reunion conference for the program Algebraic Techniques for Combinatorial and Computational Geometry.

\end{document}